\documentclass[12pt,reqno]{amsart}
  \usepackage{latexsym} 
  \usepackage[all]{xy}
  \usepackage{amsfonts} 
  \usepackage{amsthm} 
  \usepackage{amsmath} 
  \usepackage{amssymb}
  \usepackage{pifont}  
  \usepackage{enumerate}
  \usepackage{dcolumn}
  \usepackage{comment}
\newcolumntype{2}{D{.}{}{2.0}}
  \xyoption{2cell}

 \usepackage{tikz}
\usetikzlibrary{arrows}

  \def\<{{\langle}} 
  \def\>{{\rangle}} 
   
  \def\la{{\triangleright}}

  \def\note#1{{}}

  \def\note#1{} 

  \def\lhom#1#2#3{{}{{\rm Hom}\sb{#1}(#2,#3)}} 
  
  \def\hom#1#2#3{{{\rm Hom}\sb{#1}(#2,#3)}}

  \def\abs#1{\mathrm{Abs}(#1)}

  \def\beq{\begin{equation}} 
  \def\eeq{\end{equation}}

  \def\mh#1{[#1]}

  \newcounter{zlist} 
  \newenvironment{zlist}{\begin{list}{(\arabic{zlist})}{ 
  \usecounter{zlist}\leftmargin2.5em\labelwidth2em\labelsep0.5em 
  \topsep0.6ex%\itemsep0.3ex plus0.2ex minus0.3ex 
  \parsep0.3ex plus0.2ex minus0.1ex}}{\end{list}}

  \newcounter{blist} 
  \newenvironment{blist}{\begin{list}{(\alph{blist})}{ 
  \usecounter{blist}\leftmargin2.5em\labelwidth2em\labelsep0.5em 
  \topsep0.6ex %\itemsep0.3ex plus0.2ex minus0.3ex 
  \parsep0.3ex plus0.2ex minus0.1ex}}{\end{list}} 

  \newcounter{rlist} 
  \newenvironment{rlist}{\begin{list}{(\roman{rlist})}{ 
  \usecounter{rlist}\leftmargin2.5em\labelwidth2em\labelsep0.5em 
  \topsep0.6ex %\itemsep0.3ex plus0.2ex minus0.3ex 
  \parsep0.3ex plus0.2ex minus0.1ex}}{\end{list}}

\def\stac#1{\raise-.2cm\hbox{$\stackrel{\displaystyle\otimes}{\scriptscriptstyle{#1}}$}}

\def\cten#1{\raise-.2cm\hbox{$\stackrel{\displaystyle\widehat{\otimes}}
{\scriptscriptstyle{#1}}$}}

  \def\Label#1{\label{#1}\ifmmode\llap{[#1] }\else 
  \marginpar{\smash{\hbox{\tiny [#1]}}}\fi} 
  \def\Label{\label}

  \newtheorem{proposition}{Proposition}[section]
  \newtheorem{lemma}[proposition]{Lemma} 
  \newtheorem{corollary}[proposition]{Corollary} 
  \newtheorem{theorem}[proposition]{Theorem} 

\theoremstyle{definition} 
  \newtheorem{definition}[proposition]{Definition}
    
  \newtheorem{example}[proposition]{Example}

  \theoremstyle{remark} 
  \newtheorem{remark}[proposition]{Remark}

  \newcounter{c} 
   
  \newcommand{\etyk}[1]{\vspace{-7.4mm}$$\begin{equation}\Label{#1} 
  \addtocounter{c}{1}} 
  \renewcommand{\]}{\ifnum \value{c}=1 $$\else \end{equation}\fi} 
   \numberwithin{equation}{section}

\newcommand{\Mod}{{\mathbf{mod}}}

\def\NN{{\mathbb N}}

\def\ZZ{{\mathbb Z}}

\newcommand{\gG}{\mathrm{G}}
\newcommand{\hH}{\mathrm{H}}

\newcommand{\tT}{\mathrm{T}}

\newcommand{\Aa}{\mathcal{A}}

\newcommand{\Cc}{\mathcal{C}}

\newcommand{\Gg}{\mathcal{G}}
\newcommand{\Hh}{\mathcal{H}}

\newcommand{\Tt}{\mathcal{T}}

\def\*C{{}^*\hspace*{-1pt}{\Cc}}

\def\text#1{{\rm {\rm #1}}}

 \def\1{\mathbf{1}}

  \def\la{\triangleright}

\def\ahrd{\mathbf{Ah}}

\def\lto{\longmapsto}
\def\lra{\longrightarrow}
\def\lhom#1#2#3{\mathrm{Hom}_#1(#2,#3)}

\def\rev#1{#1^\circ}

    \def\wphi{\widehat{\varphi}}
\def\sym#1{\colon\!\!#1\!\colon\!\!}
\def\osym#1{{\overline{\sym{#1}}}}
\def\1\mathbf{1}
    
\def\ds{\hbox{-}}

\def\boks#1#2#3{\overset{#2}{\underset{#1}{\boxplus}}#3}
\def\Abs{\mathrm{Abs}}

\begin{document}

\title[Modules over trusses: direct sums and free modules]{Modules over trusses vs modules over rings: direct sums and free modules}
%\title{On the category of modules over a truss}

\author{Tomasz Brzezi\'nski}

\address{
Department of Mathematics, Swansea University, 
Swansea University Bay Campus,
Fabian Way,
Swansea,
  Swansea SA1 8EN, U.K.\ \newline \indent
Department of Mathematics, University of Bia{\l}ystok, K.\ Cio{\l}kowskiego  1M,
15-245 Bia\-{\l}ys\-tok, Poland}

\email{T.Brzezinski@swansea.ac.uk}

\author{Bernard  Rybo{\l}owicz}

\address{
Department of Mathematics, Swansea University, 
Swansea University Bay Campus,
Fabian Way,
Swansea,
  Swansea SA1 8EN, U.K.}

\email{bernard.rybolowicz@swansea.ac.uk}



\begin{abstract}
Categorical constructions on heaps and modules over trusses are considered and contrasted with the corresponding constructions on groups and rings. These include explicit description of free heaps and free Abelian heaps, coproducts or direct sums of Abelian heaps and modules over trusses, and description and analysis of free modules over trusses. It is shown that the direct sum of two non-empty Abelian heaps is always infinite and isomorphic to the heap associated to the direct sums of the group retracts of both heaps and $\mathbb{Z}$. Direct sum is used to extend a given truss to a ring-type truss or a unital truss (or both). Free modules are constructed as direct sums of a truss. It is shown that only free rank-one modules are free as modules  over the associated truss. On the other hand, if a (finitely generated) module over a truss associated to a ring is free, then so is the corresponding quotient-by-absorbers module over this ring.
\end{abstract}
\date\today
\maketitle

%\baselineskip=21pt

%%%%%%%%%%%%%%%%chapter 1%%%%%%%%%%%%%%%%%%%%%%%%%%%%

\section{Introduction}
Trusses and skew trusses were defined in \cite{Brz:tru} in order to capture the nature of the distinctive distributive law that characterises braces and skew braces \cite{Rum:bra}, \cite{CedJes:bra}, \cite{GuaVen:ske}. A (one-sided) truss is a set with a ternary operation which makes it into a heap or herd (see \cite{HolLaw:Wag}, \cite{Pru:the}, \cite{Bae:ein} or \cite{Sus:the}) together with an associative binary operation that distributes (on one side or both) over the heap ternary operation. If the specific binary operation admits it, a choice of a particular element could fix a group structure on a heap in a way that turns the truss into a ring or a brace-like system (which becomes a brace provided the binary operation admits  inverses). In \cite{Brz:par} the study of trusses from the ring-theoretic point of view has been initiated, in particular, the notion of a module over a truss was introduced. The present paper, intended as the first in a series which will focus on categorical properties of modules over trusses and contrast them with analogous properties of modules over rings,  is a natural continuation of these studies.

As heaps, trusses, and modules over trusses form varieties of algebras, the structure of their categories is rich and -- in many respects -- known. For example, they admit free objects (and the free-forgetful adjunction is monadic), products, coproducts, etc., see e.g.\ \cite{Ber:inv}. Thus our first aim is not to re-discover these properties but to describe them explicitly exploring close connection between heaps and groups so that they can be expressed in terms of more familiar systems with binary rather than ternary operations.

The paper is organised as follows. We start with a preliminary Section~\ref{sec.pre} in which we list basic properties of heaps, trusses and their modules.
The main goal of Section~\ref{sec.free.heap}  is to construct coproducts of Abelian heaps and modules over a truss and relate them  to coproducts of Abelian groups and modules over a ring.  We begin by  describing an explicit construction of  free heaps and  free Abelian heaps which involves grafting and pruning reduced (parity-symmetric in the Abelian case) words of odd-length, and we briefly introduce the product of heaps. It is then shown that coproducts of Abelian heaps (and then modules over a truss) are specific, explicitly described, quotients of free Abelian heaps. We list elements of coproducts, and this knowledge of the contents of the coproduct of two Abelian heaps allows us to relate it  to the direct sum of groups obtained as retracts of these heaps. Specifically we show that the direct sum (coproduct) of two Abelian heaps is isomorphic to the heap associated to the direct sum of corresponding two Abelian groups and $\ZZ$. An immediate consequence of this identification is that the direct sum of two non-empty heaps is an infinite heap, even if the heaps are finite. Equipped with coproducts we discuss and present extensions of general trusses to unital or ring-type trusses. 
 
Section~\ref{sec.free.mod} is focused on the construction of free unital modules over a unital truss as direct sums of the truss and on study of their relationship to modules over a ring. Given a unital ring $R$ we define two functors: the functor $\tT$ from the category of modules over $R$ to the category of modules over the truss $\tT(R)$ associated to $R$, and the functor $(-)_{\mathrm{Abs}}$ in the opposite direction. In contradistinction to the former which is simply based on the change of point of view (every ring can be viewed as a truss, and every module over a ring can be viewed as a module over this truss), the latter to a $\tT(R)$-module $M$ associates the retract of the quotient of $M$ by the submodule of its absorbers. Theorem~\ref{thm.1dim}, which is the main result of this section, establishes that given a ring $R$, the $\tT(R)$-module $\tT(N)$  associated to an $R$-module $N$ is free as a $\tT(R)$-module if and only if $N$ is isomorphic to $R$ as an $R$-module. In the converse direction, if $M$ is  a (finitely generated) free module over $\tT(R)$, then $M_{\mathrm{Abs}}$ is a (finitely generated) free module over $R$.

%%%%%%%%%%%%%%%%%%%%Chapter-2%%%%%%%%%%%%%%%%%%%%%%%%

\section{Preliminaries}\label{sec.pre}
\subsection{Heaps} \label{sec.heap}
Following \cite{HolLaw:Wag},\cite{Pru:the}, \cite{Bae:ein} or \cite{Sus:the}  a {\em heap} or a {\em herd} is a set $H$ together with a ternary operation $
[---]: H\times H\times H\lra H$ such that, for all $a,b,c,d,e\in H$,
\begin{equation}\label{heap.def}
[[a,b,c],d,e] = [a,b,[c,d,e]], \qquad [a,b,b] = a = [b,b,a].
\end{equation}
The first of equations \eqref{heap.def} is often referred to as the {\em associativity} the remaining two are known as {\em Mal'cev identities}. Any of the latter implies that a ternary heap  operation is an idempotent operation. A morphism of heaps  is a function that preserves ternary operations.  
A singleton set with the (unique) ternary operation is the terminal object in the category of heaps, which we denote by $\star$.  As the definition of a heap uses only universal quantifiers, the empty set with the unique ternary operation given by $\emptyset\times \emptyset\times\emptyset\lra \emptyset$ is a heap, which is the initial object in the category of heaps.

A heap $H$ is said to be {\em Abelian} if for all $a,b,c\in H$,
\begin{equation}\label{ab.heap}
[a,b,c] = [c,b,a].
\end{equation}
The full subcategory of the category of heaps consisting of Abelian heaps is denoted by $\ahrd$. Homomorphism sets of Abelian heaps are themselves Abelian heaps with the point-wise operation, i.e., for all $\varphi, \varphi',\varphi'' \in \ahrd(H,K)$, the function
\begin{equation}\label{mor.heap}
[\varphi, \varphi',\varphi'']: H\lra K, \qquad a\lto [\varphi(a), \varphi'(a),\varphi''(a)],
\end{equation} 
is a homomorphism of (Abelian) heaps and the assignment of $[\varphi, \varphi',\varphi'']$ to $\varphi, \varphi'$ and $\varphi''$ satisfies \eqref{heap.def}.

There is a close relationship between heaps and groups. Given a group $G$, there is an associated heap $\mathrm{H}(G)$ with operation, for all $g,h,k\in G$,
\begin{equation}\label{heap-form-gr}
[g,h,k] = gh^{-1}k.
\end{equation}
This heap is Abelian if the group is Abelian. A group homomorphism is a heap morphism, hence the assignment $G\lto \hH(G)$ is a functor from the category of (Abelian) groups to that of (Abelian) heaps. Conversely, given a heap $H$ and any element $e\in H$, the binary operation on $H$, defined for all $a,b\in H$,
\begin{equation}\label{gr-from-heap}
a\cdot_e b = [a,e,b],
\end{equation}
makes $H$ into a group with the neutral element $e$ and the inverse of $a\in H$, $a^{-1}= [e,a,e]$. This group is known as a {\em retract} of $(H,[---])$ (see e.g. \cite{Dud:ter}) and we denote it by $\mathrm{G}(H;e).$ One easily checks that the endomaps, defined for all $e,f\in H$,
\begin{equation}\label{swap}
\tau_e^f : H\lra H, \quad a\lto [a,e,f], \qquad \tau_f^e : H\lra H, \quad a\lto [a,f,e]
\end{equation}
are mutually inverse heap isomorphisms. Furthermore, they are isomorphisms of groups $\mathrm{G}(H;e)$ and $\mathrm{G}(H;f)$. Thus all groups associated to a given heap can be identified up to isomorphism. The process of converting heaps into groups and groups to heaps is asymmetric, which is best expressed by the following formulae
\begin{equation}\label{h-g-h-g}
\mathrm{H}\left(\mathrm{G}\left(H;e\right)\right) = H, \qquad \mathrm{G}\left(\mathrm{H}\left(G);e\right)\right) \cong \mathrm{G}\left(\mathrm{H}\left(G);f\right)\right)
\end{equation}
for all heaps $H$ and groups $G$. One can also easily observe that two heaps $H, H'$ are isomorphic if and only if $\mathrm{G}(H,e_{H})$ and $\mathrm{G}(H',e_{H'})$, for any (and hence all) choices of $e_H$, $e_{H'}$ are isomorphic. Indeed, since a group homomorphism is a morphism of corresponding heaps, if groups are isomorphic then corresponding heaps are isomorphic too.  Conversely,  if there is a heap isomorphism $\varphi:H\lra H'$, then,  for any $e_{H}\in H$,  $\varphi$ is a group isomorphism from $\mathrm{G}(H,e_{H})$ to $\mathrm{G}(H',\varphi(e_{H}))$, and since   all groups generated from the same heap are isomorphic we obtain that $\mathrm{G}(H,e_{H})$ is isomorphic with $\mathrm{G}(H',e_{H'})$, for all $e_{H}\in H$ and $e_{H'}\in H',$ as claimed.

Equations \eqref{heap.def} imply, for all $a,b,c,d,e\in H$,
\begin{equation}\label{def.heap.con}
[[a,b,c],d,e]=[a,[d,c,b],e] = [a,b,[c,d,e]].
\end{equation}
Consequently, in the case of an Abelian heap, the reduction obtained by any placement of brackets in a sequence of elements of $H$ of an odd length yields the same result. In this case we write
\begin{equation}\label{multi}
\mh{a_1,\ldots, a_{2n+1}}_n \quad \mbox{or simply}\quad \mh{a_1,\ldots, a_{2n+1}}, \qquad a_1,\ldots , a_{2n+1}\in H,
\end{equation}
for the result of applying the (Abelian) heap operation $n$-times in any possible way. Furthermore, in an Abelian heap one has at one's disposal the following {\em transposition rule},
\begin{equation}\label{tran}
[[a_1,a_2,a_3],[b_1,b_2,b_3],[c_1,c_2,c_3]] = [[a_1,b_1,c_1],[a_2,b_2,c_2],[a_3,b_3,c_3]].
\end{equation}

Also directly from equations \eqref{heap.def} one can observe that adding or removing an element in two consecutive places, whether separated by a bracket or not, does not change the value of the (multiple) heap operation. Another important consequence of the definition of a heap is that if for any $a,b\in H$ there exists $c\in H$ such that 
\begin{equation}\label{equal}
[a,b,c] = c \quad \mbox{or} \quad [c,a,b]=c,
\end{equation}
then $a=b$. In fact, in view of the Mal'cev identities,  \eqref{equal} is an equivalent characterisation of equality of elements in a heap.

A  subset $S$ of a heap $H$ that is closed under the heap operation is called a {\em sub-heap} of $H$. A sub-heap $S$ is said to be {\em normal}  if there exists $e\in S$ such that for all $a\in H$ and $s\in S$ there exists $t\in S$ such that
\begin{equation}\label{normal}
[a,e,s] = [t,e,a].
\end{equation}
Often in a statement about a heap existential quantifiers can be replaced by the  universal ones. The definition of a normal sub-heap can be stated equivalently, by requesting that for all $a\in H$ and $e,s\in S$ there exists $t\in S$ such that the equality \eqref{normal} holds. Every sub-heap of an Abelian heap is normal.

Every sub-heap $S$ of $H$ defines an equivalence relation $\sim_S$ on $H$:
\begin{equation}\label{rel.sub}
a\sim_S b \quad \mbox{if and only if} \quad \exists s\in S, [a,b,s]\in S \quad \mbox{if and only if} \quad \forall s\in S, [a,b,s]\in S.
\end{equation}
One easily checks that the equivalence class of $a$ with respect to the sub-heap relation $\sim_S$ is
\begin{equation}\label{rel.class}
\bar{a} = \{[s,t,a]\;|\; s,t\in S\}.
\end{equation}
If $S$ is a normal sub-heap of $H$, then the equivalence classes of $\sim_S$ form a heap with operation induced from that in $H$, i.e.\
\begin{equation}\label{quotient}
[\bar a, \bar b, \bar c] = \overline{[a,b,c]},
\end{equation}
where $\bar a$ denotes the class of $a\in H$, etc.  This is known as a {\em quotient heap} and is denoted by $H/S$. Note that  conditions $[a,b,s]\in S$ in \eqref{rel.sub} can be equivalently replaced by $[s,a,b]\in S$ if $S$ is a normal sub-heap. For any $s\in S$ the class of $s$ is equal to $S$.

A subset $S$ is a sub-heap of $H$ if and only if, for all $e\in S$, the retract of $S$, $\gG(S;e)$, is a subgroup of the retract $\gG(H;e)$. $\gG(S;e)$ is a normal subgroup of $\gG(H;e)$ if and only if $S$ is a normal sub-heap of $H$.

\begin{lemma}\label{lem.frac2}
Let $H$ be a heap and $S\subset H$ a normal sub-heap, then, for all $e\in S$,
$$
H/S= \mathrm{H}(\mathrm{G}(H;e)/\mathrm{G}(S;e)).
$$
\end{lemma}
\begin{proof}
In view of \eqref{rel.class} $x\in \bar{a}$ if and only if there exist $s,t\in S$ such that $x=[s,t,a]= s\cdot_et^{-1}\cdot_e a$, i.e.\ if and only if $x$ is in the coset $\gG(S,e)\cdot_e a:= \{s'\cdot_e a\;|\; s'\in S\}$. Thus the sets of equivalence classes of $\sim_S$ and the cosets in $\mathrm{G}(H;e)/\mathrm{G}(S;e)$ are mutually equal. Furthermore, since $S=\bar{e}$, for all $a,b\in H$,
$$
\bar{a}\cdot_S\bar{b} = [\bar{a},\bar{e},\bar{b}] = \overline{[a,e,b]} = \overline{a\cdot_e b},
$$
which, in the view of the aforementioned equality establishes the equality of groups 
$$
\mathrm{G}(H/S;S) = \mathrm{G}(H;e)/\mathrm{G}(S;e).
$$
The claim follows by application on $\mathrm{H}$ to both sides of this equality and \eqref{h-g-h-g}. 
\end{proof}

For any subset $X$ of a heap $H$, a sub-heap generated by $X$, denoted by $\langle X\rangle$, is equal to the intersection of all sub-heaps containing $X$. If $X$ is a singleton set, then $ \langle X\rangle = X$. If $H$ is an Abelian heap then $\langle X\rangle$ can be described explicitly as
\begin{equation}\label{gen.sub}
\langle X\rangle = \{\mh{x_1,\ldots, x_{2n+1}}\; |\; n\in \NN, x_i \in X\}.
\end{equation}

\subsection{Trusses and their modules} \label{sec.truss}
Recall from \cite{Brz:tru} or \cite{Brz:par} that a {\em truss} is an Abelian heap $T$ together with an associative binary operation (denoted by juxtaposition and called {\em multiplication}) that distributes over the heap operation, i.e., for all $s,t,t',t''\in T$,
\begin{equation}\label{truss.dist}
s[t,t',t''] = [st, st', st''] \quad \mbox{and} \quad [t,t',t'']s = [ts, t's, t''s].
\end{equation}
A truss is said to be {\em unital} or to have identity, if there is an identity for its multiplication. The identity is typically denoted by 1. A truss is said to be {\em ring-type} if there exists an element $0\in T$, called an {\em absorber} or {\em zero} such that, for all $t\in T$, $t0=0t=0$. In this case $T$ with the retract group structure $+_0$ is a ring. In the opposite direction, if $(R,+,\cdot)$ is a ring, then $(\mathrm{H}(R,+),\cdot)$ is a ring-type truss, which we denote by $\mathrm{T}(R)$ and refer to as the {\em truss associated to a ring}. If $R$ is unital, then $\mathrm{T}(R)$ is unital. Since every homomorphism of groups is a homomorphism of corresponding heaps, $\tT$ is a functor from the category of (unital) rings  into the category of (unital) trusses. Similarly, if $(B,+,\cdot)$ is a two-sided brace, then $(\mathrm{H}(B,+),\cdot)$ is a unital (but not a ring-type if $B$ is non-trivial, i.e.\ not the brace on a singleton set) truss, which we denote by $\mathrm{T}(B)$ and call the truss associated to a brace. A fundamental example of a (unital) truss is the endomorphism truss of an Abelian heap, $E(H) = \ahrd(H,H)$, which has the heap operation defined as in \eqref{mor.heap} and multiplication given by the composition of morphisms. Equivalently, $E(H)$ can be seen as a semi-direct product of any (isomorphic) group obtained from the heap operation and a chosen element of $H$ with the endomorphism monoid of this group (see \cite{Brz:par} Section 3.8).

A heap homomorphism between two trusses is a truss homomorphism if it respects multiplications. 
In case of unital trusses we require in addition that morphisms preserve identities.
In an obvious way, the terminal object $\star$ (i.e.\ the singleton set with the unique ternary operation) of the category $\ahrd$ is also a terminal object of both categories of trusses and unital trusses,  and the empty set is the initial object in the former (but not the latter).

\label{sec.mod}
Let $T$ be a truss. A {\em left $T$-module} is an Abelian heap $M$ together with an associative left action $\lambda_M: T\times M \to M$ of $T$ on $M$ that distributes over heap operations, i.e., writing $t\cdot m$ for $\lambda_M(t,m)$ one requires that, for all $t,t',t''\in T$ and $m,m',m''\in M$,
\begin{subequations}\label{module}
\begin{equation}\label{module1}
t\cdot(t'\cdot m) = (tt')\cdot m, 
\end{equation}
\begin{equation}\label{module2}
 [t,t',t'']\cdot m = [t\cdot m,t'\cdot m,t''\cdot m] , 
\end{equation}
\begin{equation}\label{module3}
 t\cdot [m,m',m''] = [t\cdot m,t\cdot m',t\cdot m''].
\end{equation}
\end{subequations}
If $T$ is a unital  truss and the action satisfies $1\cdot m = m$, then we say that $M$ is a {\em unital} or {\em normalised} module. A module homomorphism is a homomorphism of heaps between two modules that also respects the actions. The category of left $T$-modules is denoted by $T$-$\Mod$, 
and the heaps of homomorphisms between modules $M$ and $N$ are denoted by $\lhom T MN$. In fact the category of left $T$-modules is enriched over category $\ahrd$ of Abelian heaps. 
The terminal heap $\star$ with the unique possible action is the terminal object in $T$-$\Mod$, and the empty heap is the initial object.

An element $e$ of a left $T$-module $M$ is called an {\em absorber}, provided
\begin{equation}\label{absorb}
t\cdot e = e, \qquad \mbox{for all $t\in T$},
\end{equation}
i.e.\ it is invariant under the $T$-action. The set of all absorbers of a module $M$ is denoted by $\abs M$. 

If $e$ is an absorber, then the action (left) distributes over the Abelian group operation on $M$ associated to $e$ as in \eqref{gr-from-heap}, i.e.\ over the addition in $\gG(M;e)$. Note, however, that this does not mean that $\gG(M;e)$ is a module over a ring (unless $T$ is a ring-type truss).

A sub-heap $N$ of a left $T$-module $M$ is called a {\em submodule} if it is closed under the $T$-action. The $T$-action descents to the quotient heap $M/N$, making it a $T$-module with absorber $N$.

The notions of right modules are introduced symmetrically. As the left module and right module theories are completely symmetric (a right module of a given truss is a left module over the opposite truss), we discuss left modules only. Thus the term `module' means `left module' here.
 
 \section{Free heaps and coproducts of heaps and modules}\label{sec.free.heap}
All the categories discussed in the preceding section, that is categories of heaps, Abelian heaps, trusses  and their modules are varieties of algebras (in the sense of the universal algebra), hence they have free objects, limits, coproducts, euqalisers, coequalisers etc., see e.g.\ \cite{Ber:inv}, \cite{BurSan:uni}. In this section we give explicit constructions of free heaps, free Abelian heaps, coequalisers of heaps and coproducts of Abelian heaps. We start by discussing free heaps.

Let $X$ be a (non-empty) set. We define the set of {\em reduced words in $X$} as the set $W(X)$ of all odd-length words in elements of $X$ such that no consecutive letters are the same, i.e.\
$$
W(X):=\{x_{1}x_{2}\ldots x_{2n+1}\; |\; x_{i}\neq x_{i+1}\in X, n\in \NN  \}.
$$
Note that $W(X)$ is an infinite set as long as $X$ has at least two elements. Given a word $w\in W(x)$, we denote by $\rev{w}$ the opposite word, i.e.\ 
$$
\left(x_{1}x_{2}\ldots x_{2n+1}\right)^\circ =x_{2n+1}x_{2n}\ldots x_{1}.
$$
On the set $W(X)$ we define a ternary operation $[---]$ by {\em grafting and pruning}: given $u,v,w \in W(X)$, the reduced word $[u,v,w]$ is obtained by systematic removing (or pruning) all pairs of consecutive identical letters from the word $u\rev{v} w$ obtained by concatenation (or grafting) of $u$, $\rev{v}$ and $w$. Thus, in particular and for instance if $u$ is any reduced word and $w=x_{1}x_{2}\ldots x_{2n+1}$, then the step-by-step pruning process leading to $[u,w,w]$ is
$$
\begin{aligned}
u\rev{w} w &= ux_{2n+1}x_{2n}\ldots x_{1}x_{1}x_{2}\ldots x_{2n+1} \lra ux_{2n+1}x_{2n}\ldots x_{2}x_{2}\ldots x_{2n+1}\\
&\lra ux_{2n+1}x_{2n}\ldots x_{3}x_{3}\ldots x_{2n+1}
\lra \ldots \lra ux_{2n+1}x_{2n+1} = u.
\end{aligned}
$$
Note that this process is not affected by whether the word $u$ ends with any of the letters $x_i$. This shows that $[u,w,w] =u$. By similar arguments one verifies the other Mal'cev identity. Since concatenation is an associative operation and removing pairs of consecutive  identical letters of several concatenated words yields the same result irrespective of the order in which concatenated words are pruned, $[---]$ is an associative operation (in the sense of \eqref{heap.def}). Thus $(W(X), [---])$ is a heap, which we denote by $\Hh(X)$. 

\begin{lemma}\label{lem.free.heap}
The heap $\Hh(X)$ is the free heap on $X$, i.e., for any heap $H$ and  any function $\varphi: X\to H$, there exists unique filler $\wphi$ in the category of heaps of the following diagram:
$$
\xymatrix{X \ar[rr]^-{\iota_X} \ar[dr]_-\varphi&& \Hh(X) \ar@{-->}[dl]^-{\exists !\, \wphi }
\\
& H, &}
$$
where $\iota_X$ is the inclusion of $X$ into $W(X)$.
\end{lemma}
\begin{proof}
Given a function $\varphi: X\to H$, the required unique heap morphism is defined by
$$
\wphi: \Hh(X)\lra H, \qquad  x_{1}x_{2}\ldots x_{2n+1}\lto \varphi(x_{1})\varphi(x_{2})\ldots \varphi(x_{2n+1}).
$$
\end{proof}

For further convenience let us denote a free group generated by the set $X$ as $\Gg(X).$
\begin{lemma}\label{lem.heap:group}
Any free heap can be associated with a free group. Moreover $$\mathrm{H}(\Gg(X\setminus \{x\}))\cong\Hh(X),$$
where $X$ is a non-empty set and $x\in X$.
\end{lemma}
\begin{proof}
 Let $X$ be a (non-empty) set, then isomorphism needed to prove this statement is a unique filler of the diagram in Lemma~\ref{lem.free.heap}, where the function $\varphi$ is defined as follows:
$$
\varphi:X\lra \mathrm{H}(\Gg(X\setminus \{x\})), \qquad 
 y\lto 
 \begin{cases}
 y, & y\neq x\cr
  e, & y=x,
  \end{cases}
$$ 
where $e$ is the neutral element  of $\Gg(X\setminus \{x\}).$ The inverse to  $\varphi$ is given by the group homomorphism (seen as a heap homomorphism) arising from the universal property of the free group $\Gg(X\setminus \{x\})$ applied to the function 
$$
\psi :X\setminus\{x\} \lra \mathrm{G}(\Hh(X);x),\qquad 
y \lto y.
$$
\end{proof}
\begin{corollary}
Any non-empty sub-heap of a free heap is free.
\end{corollary}
\begin{proof}
Let us suppose that a non-empty sub-heap $S$ of the free heap $\Hh(X)$ is a non-free heap, then from Lemma $\ref{lem.heap:group}$ $\mathrm{G}(S;e)$ is a non-free subgroup  of $\mathrm{G}(\Hh(X),e)\cong \Gg(X\setminus\{e\})$ for some $e\in X.$ The Nielsen-Schreier theorem  \cite{Otto} states that every subgroup of a free group is free, and thus we obtain a contradiction with the assumption that $\mathrm{G}(S;e)$ is non-free, so $S$ is a free heap.
\end{proof}

\begin{example}\label{ex.free}
Let $X=\{0,1\}$ so that $\Hh(X)$ consists of all odd-length sequences of alternating digits 0 and 1. All such sequences are symmetric, hence $w^\circ =w$ and the heap operation on $\Hh(X)$ is given by concatenation and pruning.  By Lemma~\ref{lem.heap:group}  $\Hh(X)$ is isomorphic with the heap associated to a free group on a singleton set (i.e.\ on $X$ with one element removed),  so  $\Hh(X)$ is the heap associated with $\mathbb{Z}$.
\end{example}

Before we construct the coproduct of Abelian heaps, let us say a few words about the product of heaps, since as is the case in groups, rings or modules, the  coproduct of heaps is built on a product. The product of heaps $H_{1}$ and $H_{2}$ is the set $H_{1}\times H_{2}$ with operation defined component-wise, i.e.\
$$
[(h_{1},h_{2}),(h'_{1},h'_{2}),(h''_{1},h''_{2})]:=([h_{1},h_{1}',h_{1}''],[h_{2},h_{2}',h_{2}'']),
$$
for all $h_{1},h_{1}',h_{1}''\in H_{1}$ and $h_{2},h_{2}',h_{2}''\in H_{2}$.  That  $H_{1}\times H_{2}$ is the product of heaps can be proven in a way  analogous to the case of groups. It might be worth noting that
$$
H_{1}\times H_{2}\cong \mathrm{H}(\mathrm{G}(H_{1};e_{1})\times \mathrm{G}(H_{2};e_{2})),
$$
or equivalently by \eqref{h-g-h-g},
$$ 
\mathrm{G}(H_{1}\times H_{2};(e_{1},e_{2}))\cong \mathrm{G}(H_{1};e_{1})\times \mathrm{G}(H_{2};e_{2}).
$$
To prove this statement one should consider universal properties for product of groups and heaps in a similar way to the proof of Lemma~\ref{lem.heap:group}. The homomorphism 
$$
H_{1}\times H_{2}\lra \mathrm{H}(\mathrm{G}(H_{1};e_{1})\times \mathrm{G}(H_{2};e_{2}))
$$ 
is given by the universal property of the product of groups, while its inverse is constructed by the universal property of the product of heaps.

To construct free Abelian heaps we use {\em symmetric words} of odd length in alphabet $X$, 
\begin{equation}\label{sym.word.not}
w = \sym{x_1y_1x_2\ldots y_nx_{n+1}}, \qquad x_i,y_i\in X, n\in \NN,
\end{equation}
 that are defined as follows.  Each $w$ in \eqref{sym.word.not} is a set
\begin{equation}\label{sym.word}
\sym{x_1y_1x_2\ldots y_nx_{n+1}} =\{ x_{\sigma(1)}y_{\hat\sigma(1)}x_{\sigma(2)}\ldots y_{\hat\sigma(n)}x_{\sigma(n+1)}\; |\; \sigma\in S_{n+1}, \hat\sigma\in S_n\}.
\end{equation}
A symmetric word is said to be {\em reduced} if it contains only reduced words. For example, $\sym{abacd}$ is a symmetric reduced word, while $\sym{abcad}$ is not, since it contains the unreduced word $aacbd$. The set of all symmetric reduced words of odd length on $X$ is denoted by $\overline{W}(X)$. Obviously, if $\sym{w}\in \overline{W}(X)$, then $\sym{\rev{w}} = \sym{w}$. From any unreduced symmetric word one can obtain a unique symmetric reduced word by pruning. Starting with any word $x_1y_1x_2\ldots y_nx_{n+1}$ we look at all permuted words $x_{\sigma(1)}y_{\hat\sigma(1)}x_{\sigma(2)}\ldots y_{\hat\sigma(n)}x_{\sigma(n+1)}$. If any of these permuted words is not reduced, we prune it by removing pairs of consecutive identical letters. The shortest remaining word will yield the required reduced symmetric word. The heap operation on  $\overline{W}(X)$ is obtained by concatenations of representatives of symmetric reduced words followed by symmetric pruning. We use notation \eqref{sym.word.not} for both an unreduced word and the one to which it can be reduced. The resulting heap is the {\em free Abelian heap} on $X$ and is denoted by $\Aa(X)$. 

\begin{remark}\label{rem.free.ab}
One can easily employ the same isomorphism as in the proof of Lemma~\ref{lem.heap:group} to observe that the free Abelian heap on a non-empty set $X$ is isomorphic to the heap associated with the free Abelian group on $X\setminus\{x\}$, for any $x\in X$.
\end{remark}

Given Abelian heaps $A$, $B$, their {\em direct sum} or {\em coproduct} $A\boxplus B$ can be constructed as follows. Start with the free Abelian heap on the disjoint union of sets $A\sqcup B$, $\Aa(A\sqcup B),$ and apply the ternary operations of $A$ and $B$ whenever possible to reduce words further to the point when no reduction is possible. In other words, we fix  $e\in \Aa (A\sqcup B)$ and take the sub-heap $C_{e}$ of the $\Aa(A\sqcup B)$ generated by 
$$
[[a,a',a''],[a,a',a'']_{A},e],\quad [[b,b',b''],[b,b',b'']_{B},e],
$$
where $a,a',a''\in A,$ $b,b',b''\in B$,  and $[---],[---]_{A},[---]_{B}$ are ternary operations in $\Aa(A\sqcup B)$, $A$ and $B$, respectively, and consider the quotient heap $A\boxplus B=\Aa(A\sqcup B)/C_{e}.$ One can prove that this defines a congruence on $\Aa(A\sqcup B)$ the equivalence classes of which are denoted by $\osym{s_1s_2\ldots s_{2n+1}}$, $s_i\in A\sqcup B$, and which form the Abelian heap  $A\boxplus B$. More explicit  ways of describing the elements of $A\boxplus B$ are possible.

\begin{proposition}\label{prop.sum}
Let $A$ and $B$ be Abelian heaps.
\begin{zlist}
\item The direct sum $A\boxplus B$ contains only the following (types) of symmetric words in $A$ and $B$:
\begin{blist}
\item Elements $a\in A$ and $b\in B$.
\item Three letter words $\osym{a bb'}$ and $\osym{aa'b}$, with $a\neq a'\in A$ and $b\neq b'\in B$.
\item Alternating words $\osym{a_1b_1a_2\ldots a_{n}b_na_{n+1}}$ and $\osym{b_1a_1b_2\ldots b_{n}a_nb_{n+1}}$, where $a_i\in A$ and $b_i\in B$.
\end{blist}
\item Fix any $e_A$, $e_B\in B$. Then any of the multi-letter words in statement (1) can be written as  
$$
\osym{a be_B}, \; \osym{bae_A}, \; \osym{abe_Ae_B\ldots e_Ae_Be_A}, \;  \osym{bae_Be_A\ldots e_Be_Ae_B}, \quad a\in A,\, b\in B.
$$
\end{zlist}
\end{proposition}
\begin{proof}
(1) It is clear that $A\boxplus B$ contains words listed in (a) and (b) and that such words cannot be reduced any further. It is also clear that there could be no clusters of more than two consecutive letters from either $A$ and $B$. We will show that any cluster of two letters from the same alphabet can be removed from a word of length at least five. Taking into account the $A$-$B$ symmetry suffices it consider clusters $abb'a'$ with  $a,a'\in A$, $b,b'\in B$ within a symmetric word. If this word has more than five letters, then it contains an additional element of $B$. Depending on the parity of its position, it can be swapped with either $a$ or $a'$ to form a cluster of three letters in $B$ in-between $a$ and $a'$, which then is reduced to a single element by using the heap operation in $B$. In case the word has five letters, by swapping and using heap operations it can be reduced to an at most three letter word of type $abb'$ or $aa'b$.
This completes the proof.

(2)
Using the axioms of an Abelian heap and the definition of $A\boxplus B$, we can compute
$$
\osym{abb'} = \overline{[\sym{a bb'},e_B,e_B]} = \osym{a[bb'e_B]e_B} = \osym{ab''e_B},
$$
with $b'' = [b,b',e_B]$ as required. The case of $\osym{aa'b}$ is dealt with in a similar way. Words in alternating letters can be transferred to the prescribed form by consecutive applying of the above procedure. Explicitly, for $w=a_1b_1\ldots a_{n-1}$,
$$
\begin{aligned}
\osym{wb_{n-1}a_nb_na_{n+1}} &= \overline{[\sym{wb_{n-1}a_nb_na_{n+1}},e_A, e_A]} = \osym{wb_{n-1} a_nb_na_{n+1}e_A e_A} \\
&= \osym{wb_{n-1} a_ne_Aa_{n+1}b_ne_A}  = \osym{wb_{n-1}a'_nb_ne_A}\\
& = \osym{wb_{n-1}a'_nb_ne_Be_Be_A} = \osym{wb_{n-1}e_Bb_na'_ne_Be_A}
= \osym{w b'_{n-1}a'_ne_Be_A},
\end{aligned}
$$
etc., with $a'_n=[a_n,e_A,a_{n+1}]$ and $b'_{n-1}=[b_{n-1},e_B,b_{n}]$. 
\end{proof}

We refer to sequences of the alternating $e_A$ and $e_B$ as to {\em tails}. 

\begin{proposition}\label{prop.sum.uni}
Let $A$ and $B$ be Abelian heaps. Together with the inclusions $\iota_A: A\lra A\boxplus B$, $a\lto a$, and 
$\iota_B: B\lra A\boxplus B$, $b\lto b$, $A\boxplus B$ is a coproduct in the category of Abelian heaps.
\end{proposition}
\begin{proof}
We need to prove that given an Abelian heap $H$ and heap morphisms $f: A\lra H$ and $\psi: B\lra H$, there is a unique filler $\varphi\boxplus \psi$ in the diagram:
\begin{equation}\label{sum.diag}
\xymatrix{&& H && \cr A \ar[rr]^{\iota_A}\ar[urr]^\varphi & &A\boxplus B \ar@{-->}[u]_{\varphi\boxplus \psi} & & B\ar[ll]_{\iota_B}\ar[ull]_\psi .}
\end{equation}
It is clear that the unique way of defining a heap homomorphism $\varphi\boxplus \psi$ that fits diagram \eqref{sum.diag} is to set $\varphi\boxplus \psi(a) =\varphi(a)$ and $\varphi\boxplus \psi(b) = b$, for all $a\in A,b\in B$, and then extend it to words in $A\boxplus B$ letter-by-letter. We need to assure, however, that this definition is independent on the choice of representatives in the equivalence classes of symmetrised reduced words listed in, say, statement (1) of Proposition~\ref{prop.sum}. Two classes  can be equal if and only if they are of the same type (i.e.\ starting with an element of $A$ or starting with an element of $B$ as in Proposition~\ref{prop.sum}~(1)(c), or with two elements of $A$ or two elements of $B$ as in Proposition~\ref{prop.sum}~(1)(b)), as there is no way of joining elements in the same heap to produce a single element and thus reduce the length of the word or change its type. We look at these possibilities in turn.

If 
$
\osym{ab_1b_2} = \osym{a'b_1'b_2'},
$
then using the Mal'cev identity, symmetry and the definition of heap operation in $A\boxplus B$ we find
$$
\begin{aligned}
a' &= \osym{a'ab_1b_2ab_1b_2} =\osym{a'a'b'_1b'_2ab_1b_2}
= \osym{b'_1b'_2ab_1b_2} = \osym{ab'_2[b'_1,b_1,b_2]},
\end{aligned}
$$
since $A$ and $B$ are disjoint in $A\boxplus B$ and relation is given by symmetrisation and pruning this implies that $a'=a$ and $b'_2=[b'_1,b_1,b_2]$. Therefore,

$$
\begin{aligned}
(\varphi\boxplus \psi)( \osym{a'b_1'b_2'}) &= [\varphi(a'),\psi(b_1'),\psi(b_2')]
 = [\varphi(a),\psi(b_1'),[\psi(b_1'),\psi(b_1),\psi(b_2)]]\\
 & = [\varphi(a),\psi(b_1),\psi(b_2)] = (\varphi\boxplus \psi)( \osym{ab_1b_2}) ,
\end{aligned}
$$
where we used that $\psi$ is a heap morphism and the Mal'cev identity. The other case in Proposition~\ref{prop.sum}~(1)(b) follows by the $A$-$B$-symmetry.

To treat the words listed in Proposition~\ref{prop.sum}~(1)(c) we first claim that if
\begin{equation}\label{aba.eq}
\osym{a_1b_1a_2\ldots a_{n}b_na_{n+1}} = \osym{a'_1b'_1a'_2\ldots a'_{n}b'_na'_{n+1}},
\end{equation}
then
\begin{equation}\label{ab.exp}
a'_{n+1} = \mh{a_1,a_1',\ldots,a'_n,a_{n+1}} \quad \mbox{and} \quad b'_n = \mh{b_1,b_1',\ldots,b'_{n-1},b_{n}}.
\end{equation}
We prove this assertion by induction on $n$. The case of $n=1$ follows by similar reasoning as in the case already studied (simply replace $a$ by $a_{1},$ $a'$ by $a'_{1},$ $b_2$ by $a_2$ and $b'_2$ by $a'_2$, and use the corresponding arguments). Assume that the statement holds for some $n$, and assume that
$$
\osym{a_1b_1a_2\ldots a_{n+1}b_{n+1}a_{n+2}} = \osym{a'_1b'_1a'_2\ldots a'_{n+1}b'_{n+1}a'_{n+2}}.
$$
Then, first by using the Mal'cev identities, and then by the symmetry and the definition of operation in $A\boxplus B$,
$$
\begin{aligned}
\osym{a'_2b'_2a'_3\ldots a'_{n+1}b'_{n+1}a'_{n+2}} &= \osym{a_1b_1a_2\ldots a_{n+1}b_{n+1}a_{n+2}a_1'b_1'}\\
&= \osym{a_1b_1a_2\ldots b_n[a_{n+1},a_1', a_{n+2}] b_{n+1}b_1'}\\
&= \osym{a_1b_1a_2\ldots a_n [b_n ,b_1', b_{n+1}] [a_{n+1},a_1',a_{n+2}]}.
\end{aligned}
$$
As the length of the word is $2n+1$, the inductive assumption can be applied, so that
$$
\begin{aligned}
a'_{n+2} &= [a_1,a'_2,a_2,\ldots a'_{n+1}, [a_{n+1},a_1',a_{n+2}]] = \mh{a_1,a_1',\ldots,a'_{n+1},a_{n+2}},
\end{aligned}
$$
where the fact that $A$ is an Abelian herd has been used. The formula for $b'_{n+1}$ can be derived using the second part of the conjunction in the inductive assumption. This proves that \eqref{ab.exp} holds for all $n\in \NN$.

In the situation \eqref{aba.eq}, using \eqref{ab.exp}, that both $\varphi$ and $\psi$ are heap morphisms, Mal'cev identities and the Abelian nature of $A$ and $B$, one can compute
$$
\begin{aligned}
(\varphi\boxplus \psi) &\left(\osym{a'_1b'_1\ldots  b'_{n}a'_{n+1}}\right) =\\
&= \mh{\varphi(a'_1),\psi(b'_1),\ldots, \varphi(a'_n), \psi(\mh{b_1,b_1',\ldots,b'_{n-1},b_{n}}),\varphi(a'_{n+1})}\\
&= \mh{\varphi(a'_1),\psi(b'_1),\ldots ,\varphi(a'_n), \psi(b_1),\psi(b_1'),\ldots,\psi(b'_{n-1}),\psi(b_{n}),\varphi(a'_{n+1})}\\
&= \mh{\varphi(a'_1),\psi(b'_1),\psi(b_1'),\varphi(a'_2)\ldots , \psi(b_1),\varphi(a'_n),\ldots,\psi(b'_{n-1}),\psi(b_{n}),\varphi(a'_{n+1})}\\
&= \mh{\varphi(a'_1),\varphi(a'_2),\ldots , \psi(b_1),\varphi(a'_n),\ldots,\psi(b'_{n-1}),\psi(b_{n}),\varphi(a'_{n+1})}\\
&= \ldots = \mh{\varphi(a'_1),\psi(b_1),\varphi(a'_2),\ldots , \varphi(a'_n), \psi(b_n),\varphi(\mh{a_1,a_1',\ldots,a'_n,a_{n+1}})}\\
&= \mh{\varphi(a'_1),\psi(b_1),\varphi(a'_2),\ldots , \varphi(a'_n), \psi(b_n),\varphi(a_1),\varphi(a_1'),\ldots,\varphi(a'_n),\varphi(a_{n+1})}\\
&= \mh{\varphi(a'_1),\varphi(a_1'), \psi(b_1),\varphi(a'_2),\ldots , \varphi(a'_n), \varphi(a_1), \psi(b_n),\ldots,\varphi(a'_n),\varphi(a_{n+1})}\\
&= \mh{\psi(b_1),\varphi(a'_2),\ldots , \varphi(a'_n), \varphi(a_1), \psi(b_n),\ldots,\varphi(a'_n),\varphi(a_{n+1})}\\
&= \ldots = \mh{\varphi(a_1),\psi(b_1),\ldots, \varphi(a_n), \psi(b_n), \varphi(a_{n+1})}\\
&= (\varphi\boxplus \psi) \left(\osym{a_1b_1\ldots  b_{n}a_{n+1}}\right).
\end{aligned}
$$
Thus the definition of $\varphi\boxplus \psi$ is independent on the choice of the representatives in this case. The case of the alternating words starting with elements in $B$ is dealt with in a symmetric manner (or follows by the $A$-$B$ symmetry). This completes the proof of the proposition.
\end{proof}

\begin{remark}\label{rem.sum}
Note that although Abelian heaps $A$ and $B$ can be made into Abelian groups by fixing neutral elements, say $e_A\in A$ and $e_B\in B$,  the direct sum of Abelian heaps $A\boxplus B$ is not the same as the heap  associated to the direct sum of the corresponding groups, i.e. $A\boxplus B\not= \mathrm{H}(\mathrm{G}(A;e_{A})\oplus \mathrm{G}(B;e_{B}))$. Since $\varphi$ and $\psi$ are heap morphisms in the diagram \eqref{sum.diag},
there is no need for $e_A$ and $e_B$ to be mapped to the same element of $H$ that could serve for the neutral element of the induced group structure.
\end{remark}
As in the case of Abelian groups, the explicit description of the  direct sum of two Abelian heaps in Proposition~\ref{prop.sum} can be extended to families of Abelian heaps. In case of the family $(A_x)_{x\in X}$, the direct sum $\underset{x\in X}{\boxplus} A_x$, in addition to single and three letter words $\osym{a_xa'_xa_y}$, with $a_x\neq a'_x\in A_x$ and $a_y\in A_y$, $x\neq y$, consists of words of finite odd length in which neighbouring letters come from different heaps, and in which letters from the same heap, say $A_x$, are separated by odd number of letters from heaps not labelled by $x$.

The following proposition provides one with a group-theoretic description of the coproduct of Abelian heaps.

\begin{proposition}\label{prop:sum-iso}
Let $A$ and $B$ be Abelian heaps, then 
$$
A\boxplus B\cong \mathrm{H}(\mathrm{G}(A;e_A)\oplus \mathrm{G}(B;e_B)\oplus \mathbb{Z}).
$$
\end{proposition}
\begin{proof}
The functions
$$
\begin{aligned}
\varphi_A : A & \lra \mathrm{H}(\mathrm{G}(A;e_A)\oplus \mathrm{G}(B;e_B)\oplus \mathbb{Z}), \qquad a \lto (a,e_B,0) =a,\\
\varphi_B: B & \lra \mathrm{H}(\mathrm{G}(A;e_A)\oplus \mathrm{G}(B;e_B)\oplus \mathbb{Z}), \qquad b \lto (e_A,b,1)=b+1,
\end{aligned}
$$
with understanding that whenever terms are written additively in the codomain  $e_A =e_B =0$, are heap homomorphisms. By the universal property of coproducts (cf.\ the diagram in proof of Proposition~\ref{prop.sum.uni}) there exists a unique homomorphism 
$$
\varphi: A\boxplus B\longrightarrow  \mathrm{H}(\mathrm{G}(A;e_A)\oplus \mathrm{G}(B;e_B)\oplus \mathbb{Z}),
$$ 
which restricts to $\varphi_A$ on $A$ and $\varphi_B$ on $B$. In terms of words in Proposition~\ref{prop.sum}~(2) the homomorphism $\varphi$ comes out as
$$
\begin{aligned}
\varphi\left(\osym{a be_B}\right) &= a-b , \qquad  \varphi\underbrace{\left( \osym{abe_Ae_B\ldots e_Ae_Be_A}\right)}_{e_{A}\  \text{appears}\  n\text{-times}} = a-b -n,\\
\varphi\left(\osym{bae_A}\right) &= b-a+1,   \qquad \varphi\underbrace{\left(\osym{bae_Be_A\ldots e_Be_Ae_B}\right)}_{e_{B}\  \text{appears}\  n\text{-times}} = b-a +n+1. 
\end{aligned}
$$
 The inverse of $\varphi$ is the filler of the coproduct diagram in the category of groups and is determined by
$$
\varphi^{-1}: \mathrm{G}(A;e_A)\oplus \mathrm{G}(B;e_B)\oplus \mathbb{Z}\longrightarrow \; \mathrm{G}(A\boxplus B;e_{A}),
$$
$$
0\lto e_{A},\quad 1\lto e_{B},\quad   a\lto a,\quad b\lto \osym{be_Be_A}\, ,
$$
for all $a\in A$ and $b\in B$.
Therefore, since any homomorphism of groups is a homomorphism of heaps, we conclude that $\varphi^{-1}$ is a homomorphism of heaps. Clearly, compositions of $\varphi$ and $\varphi^{-1}$ give  identities so $\varphi$ is an isomorphism of heaps as required.
\end{proof}

Observe that even the coproduct of Abelian heaps is no longer a sub-heap of the product of heaps in contrast to what happens in the categories of groups.

Since the coproduct is an associative operation on a category, the identification of Proposition~\ref{prop:sum-iso} can be iterated and transferred easily to coproducts of any finite (or infinite) number of heaps.  In particular, we obtain

\begin{corollary}\label{cor:sum-free}
Let $X=\{x_{1}\ldots x_{n}\}$ be a finite set. Then
$$
 \Hh(\{x_{1}\})\boxplus\Hh(\{x_{2}\})\boxplus \ldots \boxplus \Hh(\{x_{n}\}) \cong \hH(\mathbb{Z}^{n-1})\cong \Aa(X).
$$
\end{corollary}
\begin{proof}
The free heap on a singleton set is the singleton set itself, and thus the associated (Abelian) group is the trivial group $0$. The first isomorphism thus follows from Proposition~\ref{prop:sum-iso}. The second isomorphism follows by Remark~\ref{rem.free.ab}.
\end{proof}

\begin{example}\label{ex.sum}
Let us take heaps $A=\{0_{A},1_{A}\}$ and $B=\{0_{B},1_{B}\}$ each associated with the group $C_{2}$, and choose  $0_{A}$ and $0_{B}$ as distinguished  elements of statement (2) in Proposition~\ref{prop.sum}. Proposition~\ref{prop:sum-iso} implies that $A \boxplus B\cong H(C_{2}\oplus C_{2}\oplus \mathbb{Z})$. Moreover, by choosing $G(A\boxplus B;0_{A})$ and looking at the elements from Proposition~\ref{prop.sum} we can deduce that tails of the form $0_{B}0_{A}\ldots 0_{A}0_{B}$ and $0_{A}0_{B}\ldots 0_{B}0_{A}$ represent numbers of $\mathbb{Z}$ in the direct sum.
\end{example}

Let $(A_x)_{x\in X}$ be a family of left modules over a truss $T$. By the distributivity of action, for each $t\in T$ and $X$ the function 
$
\lambda^t_x : A_{x} \lra \underset{x\in X}{\boxplus} A_x$, $a\lto t\cdot a,
$
is a homomorphism of heaps. For each $t\in T$, the family $(\lambda_x^t)_{x\in X}$ extends to the homomorphism of heaps $\underset{x\in X}{\boxplus}\lambda_x^t : \underset{x\in X}{\boxplus} A_x\lra \underset{x\in X}{\boxplus} A_x$, and thus there is a $T$-action
$$
T\times \underset{x\in X}{\boxplus} A_x\lto \underset{x\in X}{\boxplus} A_x, \qquad (t,a)\mapsto \underset{x\in X}{\boxplus}\lambda_x^t(t,a),
$$
which makes $\underset{x\in X}{\boxplus} A_x$ into a $T$-module. This action is defined letter-by-letter, so for example in the case of a two-element family of $T$-modules $A$ and $B$,
$$
t\cdot \osym{a_1b_1a_2\ldots a_{k}b_ka_{k+1}} =  \osym{(t\cdot a_1)(t\cdot b_1) (t\cdot a_2) \ldots (t\cdot a_{k})(t\cdot b_k) (t\cdot a_{k+1})},
$$
where $t\in T$, $a_i\in A$ and $b_i\in B$, etc.

For $T$-modules $A,B$ we can explicitly write out what the module action looks like on $\mathrm{G}(A;e_A)\oplus \mathrm{G}(B;e_B)\oplus \mathbb{Z}$, by transferring it through the isomorphism $\varphi$ in Proposition~\ref{prop:sum-iso}. The action  is given by the formula $t\la \varphi(x)=\varphi(t\cdot x)$, $x\in A\boxplus B$, and, for all $a\in A$, $b\in B$ and $n\in \ZZ$, it comes out as
\begin{equation}\label{action.group}
t\la (a+b+n) = t\cdot a -n(t\cdot e_A) +t\cdot b +(n-1)(t\cdot e_B) +n,
\end{equation}
where the use of the additive notation tacitly presupposes that $e_A=e_B=0$ in the direct sum of Abelian groups $\mathrm{G}(A;e_A)\oplus \mathrm{G}(B;e_B)\oplus \mathbb{Z}$.
In particular, in the case that both $t\cdot e_{B}=e_{B}$ and $t\cdot e_{A}=e_{A}$ the action takes the simple form $t\la (a,b,n)=(t\cdot a,t\cdot b,n)$.

\begin{proposition}\label{prop.unitz}
Let $T$ be a truss and $I$ be the truss on a singleton set $\{1\}$. Then $T\boxplus I$ with multiplication $\cdot$ given by
\begin{equation}\label{u.t}
1\cdot t=t\cdot 1=t \text{\ and\ }t\cdot t'=tt',
\end{equation}
where $t,t'\in T$ and $tt'$ is multiplication in  $T$ is a unital truss, which we term the {\em unital extension of $T$} and denote by $T_1$.
\end{proposition}
\begin{proof}
First note that if a binary operation defined on the heap $H$ generated by a set $X$ is associative on elements of $X$ and distributes over the heap operation, then it is associative on the whole of $H$. The operation \eqref{u.t} is associative on $T\sqcup\{1\}$ and hence it is associative on all the generators of the heap $T\boxplus I$. We need to show that this operation as defined in \eqref{u.t} can be extended to the whole of $T\boxplus I$ as a distributive operation. 
To this end, for all $s \in T\boxplus I$ consider two functions extending multiplication \eqref{u.t} to elements of $T\boxplus I$ term-by-term, i.e.\
$$
\begin{aligned}
\lambda^s_T & : T\lra T\boxplus I, \\ 
& t \lto t\cdot s := 
\begin{cases} 
\mh{t\cdot s_1, t\cdot 1, \ldots, t\cdot 1, t\cdot s_n} & \\
 ~\;\;\;\;\;\;\;\;\;\;\;= \mh{ts_1, t, \ldots, t, ts_n}, & \mbox{if $s=\mh{s_1, 1, s_2,\ldots, 1, s_n}$},\\
\mh{t\cdot 1, t\cdot s_1, \ldots,  t\cdot s_n, t\cdot 1} & \\
~\;\;\;\;\;\;\;\;\;\;\;= \mh{ t, ts_1,\ldots, ts_n, t}, & \mbox{if $s=\mh{1, s_1, 1,\ldots,  s_n, 1}$},\\
[t\cdot s_1, t\cdot s_2, t\cdot 1] = [t s_1, t s_2, t], & \mbox{if $s=[s_1, s_2,1]$},
\end{cases}
\end{aligned}
$$
where $s_i\in T$, 
and 
$$
\lambda^s_I : I\lra T\boxplus I, \qquad 1\mapsto 1\cdot s = s.
$$
The latter of these functions is a well-defined homomorphism of heaps, for all $s\in T\boxplus I$. To see that the former is so as well we first establish that its definition is independent on the presentation of $s$. If
$$
\mh{s_1, 1, \ldots, 1, s_n}=\mh{s'_1, 1, \ldots, 1, s'_n},
$$
then the Mal'cev identities imply that
$$
s'_1 = \mh{s_1, 1, \ldots, 1, s_n, s'_n,1,\ldots, 1, s'_2,1},
$$
Using the fact that $T\boxplus I$ is an Abelian heap and Mal'cev identities again, all the $1$ can be eliminated and one finds that
$$
s'_1 = \mh{s_1, s'_2, s_2,s'_3 \ldots, s_{n-1},s'_n , s_n}.
$$
Therefore,
$$
\begin{aligned}
~\mh{ts'_1, t, \ldots, t, ts'_n} &= \mh{t\mh{s_1, s'_2, s_2,s'_3 \ldots, s_{n-1},s'_n , s_n}, t, \ldots, t, ts'_n}\\
&= \mh{ts_1, ts'_2, ts_2,ts'_3 \ldots, ts_{n-1},ts'_n , ts_n, t, \ldots, t, ts'_n}\\
&= \mh{ts_1, t, \ldots, t, ts_n},
\end{aligned}
$$
by the distributive law in $T$, the Mal'cev identities and the fact that $T$ is an Abelian heap. In the second case one notices that  $\mh{1,s_1,\ldots, s_n,1} = \mh{1,s'_1,\ldots, s'_n,1}$ if and only if  $\mh{s_1,1,\ldots, 1, s_n} = \mh{s'_1,1,\ldots, 1, s'_n}$ and thus the same arguments apply. In the third case, if $[s_1, s_2,1]=[s'_1, s'_2,1]$, then $s'_1 = [s_1,s_2,s'_2]$ and again the distributive law and the Abelian heap properties imply the independence of the definition of $\lambda_T^s$ on the representation of $s$. Thus $\lambda^s_T$ is a well-defined function that is a heap morphism by the distributive law in $T$. The universal property of coproducts provides us with the unique fillers (in the category of heaps) in the following diagrams that can be considered for all $s\in T\boxplus I$:
$$
\xymatrix{&& T\boxplus I && \cr T \ar[rr]\ar[urr]^{{\lambda^s_{T}}} && T\boxplus I \ar@{-->}[u]^{\lambda^s}&& I\ar[ll]\ar[ull]_{{\lambda_{I}^s}} .}
$$
In this way the map
$$
\mu: (T\boxplus I)\times (T\boxplus I) \lra T\boxplus I, \qquad (s',s)\lto \lambda^s(s'),
$$
which extends the multiplication \eqref{u.t} to the whole of $T\boxplus I$ has been constructed. This map is a heap homomorphism in both arguments (in the first argument by the universal construction described above, in the second one by the definition of $\lambda^s_T$ and $\lambda^s_I$), that is it distributes over the heap operation in $T\boxplus I$ . This completes the proof.
\end{proof}

Note that if $T$ is a ring-type truss with absorber 0, then  its unital extension $T_1$ remains to be a ring-type truss with (the same) absorber 0. 

\begin{lemma}\label{lem.zero}
Let $T$ be a truss and let $Z$ be the truss on the singleton set $\{0\}$.  Then $T\boxplus Z $ with multiplication $\cdot$ given by
$$
0\cdot t=t\cdot 0=0 \text{\ and\ }t\cdot t'=tt',
$$
where $t,t'\in T$ and $tt'$ is multiplication in truss $T,$ is a ring-type truss, which we term the {\em ring extension of $T$} and denote by $T_0$.
\end{lemma}
\begin{proof}
The proof is analogous to that of Proposition~\ref{prop.unitz}. We only note in passing that the maps $\lambda^s_T$ are well-defined since both $Z$ and $T$ are left $T$-modules, and $\lambda^s_T$ is the action of $T$ on the direct sum of its modules.
\end{proof}

The construction in Proposition~\ref{prop.unitz} may be followed by that of Lemma~\ref{lem.zero} thus extending any truss $T$ to the unital ring-type truss $T\boxplus \{1\}\boxplus \{0\}$ (or a unital ring with the retract of the heap $T\boxplus \{1\}$ by 0 as the 
additive group). Note that any ring extension of a non-empty truss is an infinite ring, so while any ring can be interpreted as a truss, only some (and necessarily infinite at that) rings can be obtained as extensions of trusses. In particular one easily finds that $\mathrm{G}(\star\boxplus \{0\};0)$ together with the multiplication of the ring extension of $\star$ is equal to the ring of integers. Presently, we describe other examples of unital and ring extensions of trusses.

\begin{example}\label{ex.z2.c}
Let us consider the ring $\ZZ_2=\{i_{0},i_{1}\}$, where $i_0$ is the zero and $i_1$ is the identity, and the associated truss $\mathrm{T}(\ZZ_2).$  In view of the isomorphism $\varphi$ in the proof of Proposition~\ref{prop:sum-iso}, the extension of $\mathrm{T}(\ZZ_2)$ by an absorber  is
$$
\gG(\mathrm{T}(\ZZ_2)\boxplus \{0\};0):=  \{\sigma u+k i_{0}\; |\; k\in \mathbb{Z},\sigma \in \mathbb{Z}_{2}\},
$$
where $u=[i_1,i_0,0],$ $-+-=[-,0,-]$ and the appearance of $\sigma$ implies the presence or absence of $u.$ The formulae for addition and multiplication come out as:
$$
 \begin{aligned}
 (\sigma u+ki_{0})+(\sigma' u+ k'i_{0})&=(\sigma +_{(\bmod{2})}\sigma')u+  (k+k')i_{0},\\
(\sigma u+ki_{0})\cdot(\sigma' u+k'i_{0})&=\sigma\sigma'u +kk'i_0.
\end{aligned}
$$
Since $\mathrm{T}(\ZZ_2)$ is a truss with identity, so is its extension $\mathrm{T}(\ZZ_2)_0$; the identity is $i_1=u+i_0$.

\end{example}
\begin{example}\label{ex.z.c}
Let us consider the truss on the heap associated with $\mathbb{Z}$, whose multiplication is given by a constant $c$, i.e.\ $mn=c$ for all $m,n\in \mathbb{Z}.$ We  denote this truss as $\mathbb{Z}^{c}$ and describe the ring extension of $\mathbb{Z}^{c}$. To distinguish elements of $\ZZ$ from the integer multiplicities, we will use the symbols $i_m$, $m\in \ZZ$ for elements of $\mathbb{Z}^{c}$. In other words,
$$
\ZZ^c = \{i_m\; |\; m\in \ZZ\}, \qquad [i_k,i_l,i_m] = i_{k-l+m}, \quad i_mi_n = i_c.
$$
 By Proposition~\ref{prop:sum-iso}, the heap underlying $\mathbb{Z}^{c}_0 = \{0\} \boxplus \mathbb{Z}^{c}$ is isomorphic to $\hH(\mathbb{Z}\oplus \mathbb{Z})$. Following Proposition~\ref{prop.sum} we choose $0$ and $i_c\in \ZZ^c$ as special elements $e_A$ and $e_B$, respectively, and look at the retract $\mathrm{G} (\{0\} \boxplus \mathbb{Z}^{c};0)$ as the Abelian group underlying the ring $\mathbb{Z}^{c}_0$. In view of the isomorphism $\varphi$ in the proof of Proposition~\ref{prop:sum-iso},
 $$
 \mathrm{G} (\{0\} \boxplus \mathbb{Z}^{c};0) = \{\sigma i_{n}+k i_{c}\; |\; n\in\mathbb{Z}\setminus\{ c\},k\in \mathbb{Z},\sigma \in \mathbb{Z}_{2}\},
 $$
  where $-+-=[-,{0},-]$. The appearance of $\sigma$ simply indicates either the absence or presence of $i_n$. The formulae for addition and multiplication in the ring $(\mathrm{G} (\{0\} \boxplus \mathbb{Z}^{c};0), \cdot)$ come out as:
 $$
 \begin{aligned}
 (\sigma i_{n}+ki_{c})+(\sigma' i_{n'}+ k'i_{c})&=\sigma\sigma'(i_{n-c +n'} + i_{c})+ (1-\sigma')\sigma i_{n}\\
 &~\qquad \qquad +(1-\sigma)\sigma'i_{n'}+ (k+k')i_{c},\\
(\sigma i_{n}+ki_{c})\cdot(\sigma' i_{n'}+k'i_{c})&=(\sigma\sigma'+\sigma k'+\sigma' k+kk')i_{c}. 
\end{aligned}
$$

The ring extension of $\ZZ^c$ can be extended further to make it into a unital truss, $\ZZ^c_{0,1}= \{{0}\}\boxplus\mathbb{Z}^{c}\boxplus\{1\}$ as in Proposition~\ref{prop.unitz}. The corresponding retract is
$$
\mathrm{G}(\ZZ^c_{0,1};0) =  \{\sigma i_{n}+k i_{c}+l 1\; |\; n\in\mathbb{Z}\setminus\{ c\},k,l\in \mathbb{Z},\sigma \in \mathbb{Z}_{2}\}.
$$ 
The binary operations are as follows
$$
\begin{aligned}
(\sigma i_{n}+k i_{c}+ l 1)+(\sigma' i_{n'}+k' i_{c}+l' 1)& =\sigma \sigma'(i_{n-c-n']}+i_{c})+(1-\sigma')\sigma i_{n}\\ &+(1-\sigma )\sigma'i_{n'}+(k+k)'i_{c}+(l+l') 1,\\
(\sigma  i_{n}+k i_{c}+ l 1)\cdot(\sigma' i_{n'}+ k' i_{c}+ l' 1)& =(\sigma \sigma'+\sigma k'+\sigma' k+kk'+kl'+lk')i_{c}\\ &+ \sigma l'i_{n}+ \sigma'li_{n'}+ ll' 1.
\end{aligned}
$$ 
The retract $\mathrm{G}(\ZZ^c_{0,1};0)$ with multiplication $\cdot$  is a unital ring.
\end{example}

\begin{example}
Let us consider the cyclic group $C_{2}=\{a,b\}$, where $a$ is the neutral element,  with multiplication given by addition i.e.\ $a\cdot b=a+b=b$,  etc. One can observe that $C_{2}$ with such operations is a brace, and so there is the associated truss, which we denote by $\mathrm{T}(C_{2})$. This can be extended to $\mathrm{T}(C_{2})_0 = \{0\}\boxplus \mathrm{T}(C_{2})$ as in Lemma~\ref{lem.zero}. We choose $0$ and $a$ as distinguished elements and, as in the preceding example,  we study the ring structure on the retract $\mathrm{G}(\mathrm{T}(C_{2})_0;0)$. Note that  $[b,a,b] =a$ in $\mathrm{T}(C_{2})$ yields the following relation in $\mathrm{G}(\mathrm{T}(C_{2})_0;0)$,
$$
b+b = [b,0,b]= \mh{b,0,b, a,a} = [[b,a,b],0,a] = [a,0,a] = a+a.
$$
Taking this  into account we set 
$
t = [b,a,0] \in  \mathrm{T}(C_{2})_0, 
$
and  find that
$$
\mathrm{G}(\mathrm{T}(C_{2})_0;0)=\{\sigma t+n{a}\;|\;\sigma\in \mathbb{Z}_{2},n\in \mathbb{Z} \}.
$$ 
The addition and multiplication in the ring $\mathrm{G}(\mathrm{T}(C_{2})_0;0)$ come out as follows:
$$
\begin{aligned}
(\sigma  t+n{a})+(\sigma' t+n'{a}) &=(\sigma +_{(\bmod{2})}\sigma'){t}+(n+n'){a},\\
(\sigma  t+n{a})\cdot(\sigma' t+n'{a})&=\frac{1-(-1)^{\sigma'n+\sigma n'}}{2}{t}+nn'{a}.
\end{aligned}
$$
We note in passing that since $a$ is the multiplicative identity of the brace $C_{2}$,  the ring $\mathrm{G}(\mathrm{T}(C_{2})_0;0)$ also has identity $a$.
\end{example}

A few comments appear to be in order now. Examples~\ref{ex.z2.c} \& ~\ref{ex.z.c} illustrate the fact that if a truss $T$ had an absorber, making the ring extension $T_0$ does not increase the number of absorbers (this would contradict the uniqueness of absorbers), but replaces the existing absorber by a new one. The truss  $\mathbb{Z}^{c}$ has absorber $i_c$ which ceases to be an absorber in $\mathbb{Z}^{c}_0$ as $i_c(\sigma i_n +ki_c) = (\sigma +k)i_c$. 
Similar comment can be made about the unital extension: if a unital truss $T$, with identity $u$, is extended to $T_1$, then $u$ ceases to be the identity in $T_1$, as $1u=u1=u$ by the definition of the multiplication in $T_1$.  One can also notice that  the unital extension of the truss generated by a brace is no longer a truss generated by a brace (the fact that the ring extension is not a truss associated to a brace is obvious, since 0 is never an invertible element of a non-trivial ring).  The easiest example is adding identity to the truss $\star$ associated to the trivial brace $\{0\}$; $\star_1$ is a ringable truss which as a ring can be identified with $\ZZ$. Conceptually this can be understood by observing that the results of multiplication of any element from the truss associated with a brace $B$ and an element from the unital extension that does not belong to $B$ is an element of $B$ so there are no inverses in $\mathrm{T}(B)_1$ to elements in $B$.

Finally, let us observe that the ring obtained from the unital extension of the truss $\mathrm{T}(R)$ associated to a ring $R$ is the same as  the  Dorroh extension of $R$ \cite{Dor:con}.
Indeed, we know that
$\mathrm{T}(R)_1= \mathrm{T}(R)\boxplus\{1\}\cong \mathrm{H}(R\oplus\mathbb{Z})$, we can choose $0\in R$ and $1$ to be  distinguished elements and study the ring structure on the retract 
$$
\mathrm{G}(\mathrm{T}(R)_1;0) = \{r+n\;|\; r\in R, n\in \ZZ\} = R\oplus \ZZ.
$$
 Since $0$ is an absorber in $\mathrm{T}(R)$ it remains an absorber in the unital truss $\mathrm{T}(R)_1$ and we can write down the multiplication formula as
$$
\begin{aligned}(r+n)\cdot (r'+n')& =[r,0,n]\cdot[r',0,n']=[r[r',0,n'],0[r',0,n'],n[r',0,n']]\\  
&=\mh{rr',0,rn',0,nr',0,nn'}=rr'+rn'+nr'+nn'.
\end{aligned}
$$
This is precisely the multiplication rule for the  Dorroh extension of the ring $R$.

%%%%%%%%%%%%%%%%%%%%%%%chapter-4%%%%%%%%%%%%%%%%%%%%%

\section{Free modules}\label{sec.free.mod}
The aim of this section is to construct free modules over a unital truss and in the case of the truss associated to a ring study their relationship to modules over that ring. We start by introducing a candidate for a free module. Throughout this section $T$ is a unital truss.

Let $X$ be a set.  For every $x\in X,$ let us define the unital left $T$-module 
$$
Tx:=\{tx\;|\;t\in T\}, \qquad [tx,t'x,t''x] := [t,t',t'']x, \quad t\cdot (t'x) = (tt')x,
$$ 
i.e.\ $Tx$ is a $T$-module generated by $x$ that is obviously isomorphic to $T$ (as a left module). By convention $1x$ is identified with $x$, so that we may view $x$ as an element of $Tx$.
Now we can consider the direct sum module 
$$
\Tt^{X}:=\underset{x\in X}{\boxplus}{Tx}.
$$
 From Proposition~\ref{prop.sum}  and the construction of the coproduct of modules we observe that every element of $\Tt^{X}$ can be written as 
$$
\mh{t_{1}x_{1},\ldots,t_{n}x_{n},k_{i_1}x_{i_1},\ldots,k_{i_{n'}}x_{i_{n'}}}
$$
where $\{i_1,\ldots,i_{n'}\}\subseteq \{1,\ldots,n\},$ $t_{j}\in T,$ $x_{j}\in X$ and $k_{j}x_{j}=\mh{\underbrace{x_{j},e,x_{j},e,\ldots,x_{j}}_{x_{j}-appears\  k_{j}\  times}}$,   for any fixed  $e\in X$.
Moreover one can observe that there are isomorphisms of heaps 
$$
\begin{aligned}
\Tt^{X}&\cong \mathrm{H}\left(\mathrm{G}\left(Te;e\right)\oplus\left(\bigoplus_{x\in X\setminus\{e\}}\left(\mathrm{G}\left(Tx;x\right)\oplus \mathrm{G}\left(\Hh(\{ x \});x\right)\right)\right)\right)\\
&\cong \mathrm{H}\left(\mathrm{G}\left(Te;e\right)\oplus\left(\bigoplus_{x\in X\setminus\{e\}}\left(\mathrm{G}\left(Tx;x\right)\oplus \ZZ\right)\right)\right), 
\end{aligned}
$$
analogous to those found in Proposition~\ref{prop:sum-iso} and Corollary~\ref{cor:sum-free}. Although the $T$-module structure of $\Tt^X$ can be transferred to the right hand side through this isomorphism, the form of the transferred action interacts nontrivially and often in a not necessarily illuminating manner with the direct sum of groups (compare \eqref{action.group} in the case of two modules). 

Following the categorical idea of freeness (see e.g.\ the universal property in Lemma~\ref{lem.free.heap}) let us fix set $X$ and consider inclusion $\iota_{X}:X\to \Tt^X$, given by $\iota_{X}(x)=1x$, for all $x\in X$.  Then, for any unital $T$-module $N$ and any function $\varphi: X\lra N$ we obtain the following commutative diagram  
$$
\xymatrix{X \ar[rr]^-{\iota_X} \ar[dr]_-\varphi&& \Tt^{X} \ar@{-->}[dl]^-{\exists !\, \wphi }
\\
& N, &},
$$
where the unique $T$-module morphism $\wphi$ is defined by
$$
\wphi: \Tt^{X}\lra N, 
$$
$$
\begin{aligned}
\mh{t_{1}x_{1},\ldots,t_{n}x_{n},k_{i_1}x_{i_1},\ldots,k_{i_{n'}}&x_{i_{n'}}}\\
&\lto \mh{t_{1}\varphi(x_{1}),\ldots,t_{n}\varphi(x_{n}),k_{i_1}\varphi(x_{i_1}),\ldots,k_{i_{n'}}\varphi(x_{i_{n'}})},
\end{aligned}
$$
where $k_{j}\varphi(x_{j})=\mh{\varphi(x_{j}),\varphi(e),\varphi(x_{j}),\varphi(e),\ldots,\varphi(x_{j})}.$
Since this is the 
universal property characterising  a free object in the category of unital modules over $T$,  $\Tt^{X}$ is the free unital $T$-module on $X$, as expected.

Following the usual ring-theoretic conventions we can formulate
\begin{definition}\label{def.gen.mod}
A unital $T$-module is said to be {\em generated by a set $X$}, if there exists a $T$-module epimorphism $\Tt^X\lra M$. It is said to be {\em finitely generated} if there exists finite $X$ that generates $M$. $M$ is a (finitely generated) {\em free} $T$-module if it is isomorphic to $\Tt^X$, for some (finite) $X$.
\end{definition}

As in the case of modules over a ring, one can try to characterise free modules by the existence of a basis.

\begin{definition}\label{def.lin.ind}
Let $M$ be a left $T$-module and let $X$ be a non-empty subset of $M$. Let, for all $x\in X$, $\sigma_x$ denote  the left $T$-module homomorphism
$$
 \sigma_x : T \lra M, \qquad t\lto tx.
 $$
 \begin{zlist}
 \item We say that the set $X$ is {\em free} if, for all finite subsets $S$ of $X$, the map
 $\underset{x\in S}{\boxplus}\sigma_{x}$ is a monomorphism.
 \item A free set  $B$ is said to be a {\em basis} for $M$ if the map  $\underset{x\in B}{\boxplus}\sigma_{x}$ is an epimorphism.
 \end{zlist}
 \end{definition}
 
 \begin{lemma}\label{lem.intersec}
 If $X$ is a free subset of $M$, then, for all finite non-empty proper subsets $Y$ of $X$ and all $x\in X\setminus Y$,
 $$
 \sigma_x(T) \cap \left(\underset{y\in Y}{\boxplus}\sigma_{y}\right)(\boks{Y}{}{T}) = \emptyset.
 $$
 \end{lemma}
 \begin{proof}
 Set $V = \underset{y\in Y}{\boxplus}T$ and $\sigma_Y = \underset{y\in Y}{\boxplus}\sigma_{y}$, and suppose that there exist $m\in M$, $t\in T$ and $v\in V$  such that $\sigma_x(t) = \sigma_Y(v)=m$. Take any $v'\in V$. In view of Proposition~\ref{prop.sum} the words $\osym{tvv'}$ and $\osym{vtv'}$ are different, but
 $$
 \begin{aligned}
 (\sigma_x\boxplus \sigma_Y) (\osym{tvv'}) &= [\sigma_x(t),\sigma_Y(v),\sigma_Y(v')] = [m,m,\sigma_Y(v')] = \sigma_Y(v'),\\
 (\sigma_x\boxplus \sigma_Y) (\osym{vtv'}) &= [\sigma_Y(v),\sigma_x(t),\sigma_Y(v')] = [m,m,\sigma_Y(v')] = \sigma_Y(v'),
 \end{aligned}
 $$
 which contradicts the assumption that $\sigma_x\boxplus \sigma_Y$ is a monomorphism.  
  \end{proof}
 
 The statement of Lemma~\ref{lem.intersec} is in perfect categorical accord with what might be expected of a free or a linearly independent set. Just as in the case of modules or vector spaces, the intersection of the module spanned by any finite subset of a free set with a cyclic module generated by an element from within the free set but without this subset is the zero module, i.e.\ the initial object in the category of modules, so is the corresponding intersection in the case of modules over a truss  -- the empty set, i.e.\ the initial object in the category of such modules.

\begin{lemma}
Let $M$ be a left module over a truss $T$. Then $M$ is a free $T$-module if and only if $M$ has a basis. 
\end{lemma}

\begin{proof}
If $M$ has a basis $B$, then $\underset{x\in B}{\boxplus}\sigma_{x}$ is an epimorphism. Since all elements of $\Tt^B$ have finite length, i.e.\ every element belongs to  $\underset{S}{\boxplus}T$ for a finite subset $S$ of $B$, and for all such subsets $\underset{x\in S}{\boxplus}\sigma_{x}$ is one-to-one, then so is $\underset{x\in B}{\boxplus}\sigma_{x}$. Thus $\underset{x\in B}{\boxplus}\sigma_{x}$ is an isomorphism, and hence $M$ is free.

In the converse direction,  since $M$ is free there exist a set $X$ and a left $T$-module isomorphism $\Theta: \underset{x\in X}{\boxplus}Tx \lra M$. For all $x\in X$, let $e_x = \Theta(x) \in M$, and let $B=\{e_x\;|\; x\in X\}$. Since $\Theta$ is a homomorphism of $T$-modules, for all $t\in T$,
$$
\Theta(tx) = t\Theta(x) =t\cdot e_x = \sigma_{e_x}(t).
$$
Since $X$ is isomorphic to $B$, by the universality of direct sums there is a $T$-module isomorphism $\varphi: \underset{B}{\boxplus} T\lra \Tt^X$ and thus we  obtain the following commutative diagrams, for all $e_x\in B$,
$$
\xymatrix{T \ar[rr]\ar[rrd]_{\sigma_{e_x}} && \underset{B}{\boxplus} T\ar[rr]^\varphi \ar[d]^{\underset{e_x\in B}{\boxplus}\sigma_{e_x}} && \Tt^X \ar[lld]^{\Theta}\\
&& M. &&}
$$
Thus $\underset{e_x\in B}{\boxplus}\sigma_{e_x}$ is an isomorphism, and hence $B$ generates $M$. Since $\underset{e_x\in B}{\boxplus}\sigma_{e_x}$  is a monomorphism on $\underset{B}{\boxplus} T$, it is a monomorphism on each finitely generated submodule of $\underset{B}{\boxplus} T$, in particular $\underset{e_x\in S}{\boxplus}\sigma_{e_x} = \underset{e_x\in B}{\boxplus}\sigma_{e_x}\!\!\mid_{\underset{S}{\boxplus} T}$ is a monomorphism for all finite subsets $S$ of $B$. Hence $B$ is a basis for $M$.
\end{proof}

The remainder of this section is devoted to the comparison of modules over a ring with modules over the truss constructed from this ring. Recall that $\tT(R)$ denotes the truss associate with  a ring $R$.  If $M$ is a left module over a ring $R$ it is automatically a left module over the truss $\tT(R)$, and when viewed as such (with the heap operation coming from the additive group $M$), it will be denoted by $\tT(M)$. This defines a functor
$$
\tT: R\ds\Mod \lra \tT(R)\ds\Mod, \qquad M\lto \tT(M)=M, \qquad \varphi\lto \tT(\varphi) =\varphi.
$$
 Note, that a module over $\tT(R)$ is not necessarily a module over $R$, for example the heap $\hH(\ZZ)$ with the $\ZZ$-action, $m\cdot n = n$, for all $m,n\in \ZZ$ is a left module over $\tT(\ZZ)$, but not over the ring $\ZZ$. The forthcoming Lemma~\ref{lem.abs} clarifies when a module over the truss associated to a ring is a module over this ring.  Before we state this lemma, however, we make an observation about a striking difference between  free modules over a ring and free modules over the associated truss. We note, in particular, that the functor $\tT$ does not preserve freeness.
 
 \begin{example}
Let us consider module $\tT(\mathbb{Z}_{n}\oplus \mathbb{Z}_{n})$ over $\mathrm{T}(\mathbb{Z}_{n})$, for any $n>1$.  Suppose that $\tT(\mathbb{Z}_{n}\oplus \mathbb{Z}_{n})$  is a free module, i.e.\ that it is isomorphic to a direct sum of $k$-copies of $\mathrm{T}(\mathbb{Z}_{n})$. By Proposition~\ref{prop:sum-iso}, if $k>1$ then such a direct sum would be an infinite set, so it cannot be isomorphic to a module built on a finite set. Thus $k=1$, and simple element counting forces $n=n^2$, which contradicts the assumption that $n>1$. Thus  $\tT(\mathbb{Z}_{n}\oplus \mathbb{Z}_{n})$ over $\mathrm{T}(\mathbb{Z}_{n})$ is not free, despite that fact that $\mathbb{Z}_{n}\oplus \mathbb{Z}_{n}$ is a free $\mathbb{Z}_{n}$-module.
\end{example}

 Recall that an element $e$ of a $T$-module $M$ is called an {\em absorber}, if $t\cdot e =e$, for all $t\in T$. The set of all absorbers  of $M$ is denoted by $\abs M$.
 
 \begin{lemma}\label{lem.abs} 
 Let $T$ be a (unital) truss and $R$ a (unital) ring.
 \begin{zlist}
 \item The assignment:
$$
\mathrm{Abs}: T\ds\Mod \lra T\ds\Mod, \qquad M\lto \abs M, \quad \varphi\lto\varphi,
$$
is a functor. 
\item Let $M$ be a left module over $T(R)$. Then:
\begin{rlist}
 \item $\abs M = \{0\cdot m\; |\; m\in M\}$;
 \item $M=\tT(N)$ for some module of $N$ if and only if $\abs M$ is a singleton set.
  \end{rlist}
\item Let $N,N'$ be left $R$-modules. Then $N$ is isomorphic to $N'$ if and only if $\tT(N)$ is isomorphic to $\tT(N')$ as  $\tT(R)$-modules.
\item Let $M$ be a (unital) $\mathrm{T}(R)$-module.  Then
$\mathrm{G}(M/\abs M; \abs M)$ is a (unital)  $R$-module. We denote this $R$-module by $M_\mathrm{Abs}$.
\item The assignment
$$
\begin{aligned}
(-)_\mathrm{Abs} &: \tT(R)\ds\Mod \lra R\ds\Mod,\\
& M\lto M_\mathrm{Abs}, \quad (\varphi: M\to M')\lto (\varphi_\mathrm{Abs}: \overline{m}\mapsto \overline{\varphi(m)}),
\end{aligned}
$$
is a functor such that, for all $R$-modules $N$, $\tT(N)_\mathrm{Abs}\cong N$.
\item The functor $(-)_\mathrm{Abs}$ is the left adjoint to the functor $\tT$.
\item The functor $(-)_\mathrm{Abs}$ preserves monomorphisms.
\end{zlist}
\end{lemma}
\begin{proof}
(1) The distributive law over the heap operation in a $T$-module $M$ ensures that $\abs M$ is a sub-heap of $M$. That $\abs M$ is closed under the action follows immediately form the definition of an absorber. Since morphisms of $T$-modules preserve the $T$-action they also map absorbers into absorbers.

(2) (i) Since $r0 =0$ in $\tT(R)$, all elements listed are absorbers. If $e$ is an absorber, then, by the absorption property $0\cdot e = e$.

(2) (ii) If $M= \tT(N)$, then, by distributive laws for modules over rings $0$ is an absorber in $M$ and $0\cdot m = 0 \in M$, for all $m\in M$, which implies that 0 is the unique absorber of $M$. In converse direction, by (i) we know that the unique absorber is $e = 0\cdot m$.  Then one easily checks that $\gG(M;e)$ with the original action of $\tT(R)$ on $M$ is a left $R$-module.

(3) Since $\tT$ is a functor, if $N\cong N'$, then $\tT(N)\cong \tT(N')$. Conversely, since, by statement (2) both $\tT(N)$ and $\tT(N')$ have unique absorbers (they are neutral elements for addition), and a morphism of modules over a truss maps absorbers into absorbers (cf.\ statement (1)), any morphism of $\tT(R)$-modules $\tT(N) \lra \tT(N')$ is automatically a morphism of Abelian groups and hence $R$-modules.

(4) Since $\abs M$ is a submodule of $M$ by assertion (1), $M/\abs M$ is a $\tT(R)$-module with an absorber $\abs M$ (see \cite[Section~4]{Brz:par}). There are no other absorbers in $M/\abs M$, for since $M/\abs M$ is a module of $\tT(R)$, by  statement (2)(i) all its absorbers have the form $0\cdot \bar{m} = \overline{0\cdot m} = \abs M$. Thus, by statement (2)(ii), $\mathrm{G}(M/\abs M; \abs M)$ is a left $R$-module. The unitality condition is provided by the unitality of the $\mathrm{T}(R)$-module $M$.

(5) The function $\varphi_\mathrm{Abs}$ is well defined by statement (1), as $\varphi$ maps absorbers to absorbers. By the same arguments as in the proof of statement (3) $\varphi_\mathrm{Abs}$ is a homomorphism of $R$-modules. Since $\abs{\tT(N)} = \{0\}$. The elements of  $\tT(N)/\{0\}$ are all singleton subsets of $N$, $\tT(N)/\{0\} = \{\{n\} \;|\; n\in N\}$, and the stated isomorphism is simply $\{n\}\lto n$.

(6) Let $N$ be a left $R$-module and $M$ a left $\tT(R)$-module, and consider the maps:
$$
\begin{aligned}
\Theta_{M,N}: \hom R{M_\mathrm{Abs}}N \lra \hom {\tT(R)}{M}{\tT(N)}, &\qquad \varphi\lto [m\mapsto \varphi(\bar{m})],\\
\Theta^{-1}_{M,N}: \hom {\tT(R)}{M}{\tT(N)}  \lra  \hom R{M_\mathrm{Abs}}N, &\qquad \psi\lto [ \bar{m} \mapsto \psi(m)],
\end{aligned}
$$
that are clearly mutual inverses. While $\Theta_{M,N}$ is obviously well-defined, we need to establish whether the definition of $\Theta^{-1}_{M,N}$ does not depend on the choice of the representative. Suppose that $m'$ and $m$ belong to the same class. In view of the description on  $\abs M$ in (2)(i) this means that there exist $v,w\in M$ such that
$[m,m',0v]=0w$. Applying $\psi$ to this equality and using the fact that $\psi$ is a homomorphism of $\tT(R)$-modules we find that $[\psi(m),\psi(m'),0\psi(v)]=0\psi(w)$. Both $0\psi(v)$ and $0\psi(w)$ are absorbers in $\tT(N)$, but, by (2)(ii) there is exactly one absorber in $\tT(N)$, so $0\psi(v)=0\psi(w)$, and we conclude that $\psi(m)= \psi(m')$. Thus the definition of $\Theta^{-1}_{M,N}$ does not depend on the choice of the representative in the class of $m$. Checking the naturality of $\Theta_{M,N}$ is straightforward.

(7) Let $\varphi:M\to M'$ be a monomorphism of $\tT(R)$-modules. Observe that if $\varphi(w) \in \abs{M'}$, then, for all $t\in \tT(R)$, $\varphi(w) = t\varphi(w) = \varphi(tw)$ so that $w=tw$ since $\varphi$ is one-to-one. Hence  $w\in \abs{M}$. 
Assume that 
 $\varphi_{\Abs}(\overline{a})=\varphi_{\Abs}(\overline{b})$, that is $\varphi(a)\sim_{\abs{M'}} \varphi(b)$.
Since, by (1), for all $w\in \abs{M}$, $\varphi(w)\in \abs{M'}$, 
$$
 \varphi\left([a,b,w]\right) = [\varphi(a),\varphi(b),\varphi(w)] \in \abs{M'}.
$$
Thus, for all $w\in \abs M$,  $[a,b,w] \in \abs{M}$, i.e., $\overline{a}=\overline{b}$. 
 Therefore, $\varphi_{\Abs}$ is a monomorphism, as required.
\end{proof}

With the help of Lemma~\ref{lem.abs} we can prove the main result of this section, which explains the interplay between the freeness and the functor $\tT$.

\begin{theorem}\label{thm.1dim}
Let $R$ be a unital ring.
\begin{zlist}
\item For a left $R$-module $N$, $\tT(N)$ is a free $\mathrm{T}(R)$-module if and only if $N\cong R$. 
\item  If $M$ is a free module over $\mathrm{T}(R)$, then $M_\mathrm{Abs}$ is  a free $R$-module.
\end{zlist}
\end{theorem}
\begin{proof}
(1) If $N\cong R$, then $\tT(N)\cong \tT (R)$ by Lemma~\ref{lem.abs}~(3) (or simply by the fact that $\tT$ is a functor). 
In the opposite direction, assume that there exists a set $X$ such that $\mathrm{T}(N)\cong \underset{x\in X}{\boxplus}{\tT(R)x}$. By Lemma~\ref{lem.abs}~(2)(ii), $\underset{x\in X}{\boxplus}{\tT(R)x}$ must have exactly one absorber. This is the case when $X$ is a singleton set, which yields the isomorphism $N\cong R$ by Lemma~\ref{lem.abs}~(3). If $X$ has more than one element, then there exist $x,y\in X$ such that $x\not=y$ and thus $0x\not=0y \in \underset{x\in X}{\boxplus}{\tT(R)x}$ are different absorbers, which contradicts statement (2)(ii) in Lemma~\ref{lem.abs}.

(2) Assume that $M\cong \underset{x\in X}{\boxplus}{\tT(R)x}$, for some set $X$. By  Lemma~\ref{lem.abs}~(6), $(-)_{\mathrm{Abs}}$ has a right adjoint and thus it preserves coproducts, so that
$$
M_{\mathrm{Abs}} \cong \left(\underset{x\in X}{\boxplus}{\tT(R)x}\right)_{\mathrm{Abs}}\cong \underset{x\in X}{\bigoplus}{(\tT(R)x)_{\mathrm{Abs}}}\cong \underset{x\in X}{\bigoplus}Rx,
$$
where the last isomorphism follows by Lemma~\ref{lem.abs}~(5). Therefore, $M_{\mathrm{Abs}}$ is a free $R$-module as stated.

\end{proof}

Although Theorem~\ref{thm.1dim}  states that the functor $(-)_{\Abs}$ preserves freeness the proof neither gives an insight into the process of obtaining the free $R$-module nor does it explain fully 
the idea behind the definition of a basis  in Definition~\ref{def.lin.ind}.
 Let us discuss this matter further in the finitely generated case. 
 Let $X=\{x_1,\ldots, x_n\}$ be such that $M\cong \underset{x\in X}{\boxplus}{\tT(R)x}$. First we describe the submodule $\abs M$, freely identifying $M$ with the direct sum of $n$ copies of $T$. By Lemma~\ref{lem.abs}(2)(i), $\abs M = \{0\cdot m\; |\; m\in M\}$. Since $\tT(R)x_i = \{rx_i\;|\; r\in R\}$,  every element of $M$ is of the form 
$m= \osym{(r_1x_{i_1}) (r_2 x_{i_2})\ldots (r_{2k+1} x_{i_{2k+1}})}$ and hence $0\cdot m = \osym{(0x_{i_1}) (0x_{i_2})\ldots ( 0x_{i_{2k+1}})}$. Therefore $\abs M$ is the submodule of $M,$ or, more precisely $\abs {\underset{x\in X}{\boxplus}{\tT(R)x}}$, is a submodule of $\underset{x\in X}{\boxplus}{\tT(R)x}$ generated by $\{0x_1,\ldots, 0x_n\}$ as a heap. Choosing the $0x_i$ as special elements in $\tT(R)x_i$ as in (the multi-heap versions of) Proposition~\ref{prop.sum}, $\abs {\underset{x\in X}{\boxplus}{\tT(R)x}}$ is simply the sub-heap of tails, i.e.\
$$
\abs M\cong \abs {\underset{x\in X}{\boxplus}{\tT(R)x}}\cong \Hh(\{0x_{1}\})\boxplus\Hh(\{0x_{2}\})\boxplus \ldots \boxplus \Hh(\{0x_{n}\}) \cong \hH(\mathbb{Z}^{n-1});
$$
see Corollary~\ref{cor:sum-free}. By (the multi-heap extension of Proposition~\ref{prop:sum-iso})
$$
M\cong \underset{x\in X}{\boxplus}{\tT(R)x} \cong \hH\left( \bigoplus_{i=1}^n \gG(\tT(R)x_i;0x_i)\oplus \ZZ^{n-1} \right)
 \cong \hH\left(R^n\oplus \ZZ^{n-1}\right).
$$
Since the $\ZZ^{n-1}$-part arises from tails made of the absorbers $0x_i$, the action of $T(R)$ on this part is trivial, i.e.\ the $\tT(R)$ action on $M$ transfers to 
$$
r\cdot (r_1,r_2,\ldots , r_n, k_1,\ldots, k_{n-1}) = (rr_1,rr_2,\ldots , rr_n, k_1,\ldots, k_{n-1}), \qquad r,r_i\in R, k_i\in \ZZ.
$$
Putting all this together yields an isomorphism of  $R$-modules, 
$$
M_{\mathrm{Abs}} = \gG \left(M/\abs M; \abs M\right)\cong \left(R^n\oplus \ZZ^{n-1}\right)/\mathbb{Z}^{n-1} \cong R^n,
$$
where the first isomorphism follows by Lemma~\ref{lem.frac2}, so that $M_{\mathrm{Abs}} $ is a free module. 

Now, assume that $M$ is a free rank $n$ module over $\tT(R)$ with a basis $B.$  To prove that $B$ is a basis for $M_{\Abs}$ observe that, for all $S\subset B,$ $(\boks{s\in S}{}{\sigma_s})_{\Abs}$ is a monomorphism of $R$-modules(see Lemma \ref{lem.abs}~(7)), and by the discussion following Lemma~\ref{lem.intersec}, appropriate intersections of $(\boks{s\in S}{}{\sigma_s})_{\Abs}(\boks{S}{}{\tT(R)})_{\Abs}$ are no longer empty; they are now the initial object of 
$R\ds\Mod$, i.e.\ $\{0\}.$ Firstly, since $B$ spans $M,$ then it also spans $M_{\Abs}.$ Therefore, it is enough to show that the set $B$ is linearly independent in $M_{\Abs}.$ Suppose to the contrary that $B$ is linearly dependent,  so that there exist $r_{i}\in R$ such that 
$$
r_1b_1+\ldots+r_{n}b_{n}=0,
$$
for $b_i\in B,$ and $r_{n}\not=0.$ This implies that $r_1b_1+\ldots+r_{n-1}b_{n-1}=-r_{n}b_{n}.$ Furthermore
$$
-r_{n}b_{n}\in (\boks{i=1}{n-1}{\sigma_{b_i}})_{\Abs}(\boks{i=1}{n-1}{\tT(R))_{Abs}}\cap(\sigma_{b_{n}})_{\Abs}((\tT(R))_{\Abs})=\{0\}.
$$
Therefore, $-r_nb_n=0$ and $(\sigma_{b_{n}})_{\Abs}(0)=(\sigma_{b_{n}})_{\Abs}(-r_n),$ and since $(\sigma_{b_{n}})_{\Abs}$ is a monomorphism, $r_n=0.$ Now by recursion for all $i=1,\ldots,n,$ $r_i=0,$ and we arrive at a contradiction with the assumption that $B$ is a linearly dependent set. Therefore, $B$ is a basis for $M_{\Abs}.$ To sum up,  at least in the case of the truss associated to a ring, Definition~\ref{def.lin.ind} of a free set is justified by the linear independence of its elements in the associated module over a ring.

\end{document}